\documentclass[11pt]{article}
\usepackage{amsmath,amsthm,amsfonts,amssymb,url,graphicx,xcolor,authblk,longtable}
\usepackage{tikz}
\usepackage[title]{appendix}
\usepackage{enumitem}
\usetikzlibrary{calc}
\usetikzlibrary{arrows,decorations.markings}

\makeatletter
\g@addto@macro\bfseries{\boldmath}
\makeatother

\newcommand{\D}[3]{\ensuremath{\mathcal{D}_{#1}#2#3}}

\newtheorem{theorem}{Theorem}[section]
\newtheorem{lemma}[theorem]{Lemma}
\newtheorem{corollary}[theorem]{Corollary}

\newtheorem{example}{Example}[section]

\newcommand{\A}{\ensuremath{\mathcal{A}}} 
\newcommand{\B}{\ensuremath{\mathcal{B}}} 
\newcommand{\C}{\ensuremath{\mathcal{C}}} 
 
\newcommand{\K}{\ensuremath{K_{\mathsf{DTS}}}}
\newcommand{\zed}{\ensuremath{\mathbb{Z}}}

\definecolor{red}{RGB}{220,20,60}

\newcommand{\STS}{\ensuremath{\mathsf{STS}}} 
\newcommand{\DTS}{\ensuremath{\mathsf{DTS}}} 
\newcommand{\TTS}{\ensuremath{\mathsf{TTS}}} 
\newcommand{\SAT}{\textsf{2-SAT} }

\title{Block-avoiding point sequencings of directed triple systems}
\author[1]{Donald L.\ Kreher}
\author[2]{Douglas R.\ Stinson%
\thanks{D.R.\ Stinson's research is supported by  NSERC discovery grant RGPIN-03882.}}
\author[3]{Shannon Veitch}
\affil[1]{Department of Mathematical Sciences, 
Michigan Technological University 
Houghton, MI 49931,  
U.S.A.}
\affil[2]{David R.\ Cheriton School of Computer Science, University of Waterloo,
Waterloo, Ontario, N2L 3G1, Canada}
\affil[3]{Department of Combinatorics and Optimization, University of Waterloo,
Waterloo, Ontario, N2L 3G1, Canada}

\begin{document}
\maketitle

\begin{abstract}
A  \emph{directed triple system of order $v$} (or, \DTS$(v)$) is decomposition of the complete directed graph $\vec{K_v}$
into transitive triples. A \emph{$v$-good sequencing}  of a \DTS$(v)$  is a permutation of the points of the design, say $[x_1 \; \cdots \; x_v]$, such that, for every triple $(x,y,z)$ in the design,
it is \emph{not} the case that $x = x_i$, $y = x_j$ and $z = x_k$ with $i < j < k$.
We prove that there exists a \DTS$(v)$ having a $v$-good sequencing for all positive integers $v \equiv 0,1 \bmod {3}$. Further, for all positive integers $v \equiv 0,1 \bmod {3}$, $v \geq 7$, we prove that there is a \DTS$(v)$ that does not have  a $v$-good sequencing. We also derive some computational results concerning  $v$-good sequencings of all the nonisomorphic \DTS$(v)$ for $v \leq 7$.
\end{abstract}

\section{Introduction}
\label{intro.sec}

A \emph{Steiner triple system of order $v$} is a pair $(X, \B)$, where 
$X$ is a set of $v$ \emph{points} and $\B$ is a set of 3-subsets of $X$ (called 
\emph{blocks}), such that every pair of points occurs in exactly one block.
We will abbreviate the phrase ``Steiner triple system of order $v$'' to 
\STS$(v)$. It is well-known that an STS$(v)$ contains exactly $v(v-1)/6$ blocks, and an 
STS$(v)$ exists if and only  if $v \equiv 1,3 \bmod 6$. 
The definitive reference for Steiner triple systems is the book \cite{CR} by Colbourn and Rosa.

There are two directed variants of \STS$(v)$, which are known as Mendelsohn triple systems and directed triple systems. 
We study directed triple systems in this paper. First, we define a \emph{transitive triple} to be an ordered 
triple $(x,y,z)$, where $x,y,z$ are distinct. 
This triple contains the directed edges $(x,y)$, $(x,z)$ and $(y,z)$
(we might also write these directed edges as $xy$, $xz$ and $yz$, respectively). 
Thus, the triple $(x,y,z)$ can be thought of as the following directed graph:

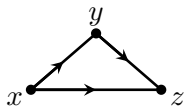
\begin{center}
\tikz[scale=0.5,line width=1]{
\coordinate (x) at (210:2.0);
\coordinate (y) at ( 90:0.5);
\coordinate (z) at (330:2.0);
\draw[->,>=stealth] (x)--($(x)!0.5!(y)$);
\draw[->,>=stealth] (y)--($(y)!0.5!(z)$);
\draw[->,>=stealth] (x)--($(x)!0.5!(z)$);
\draw(x)--(y)--(z)--cycle;
\node at ($(0,0)!1.25!(x)$) {$x$};
\node at ($(0,0)!1.90!(y)$) {$y$};
\node at ($(0,0)!1.25!(z)$) {$z$};
\foreach \i in {x,y,z}
{
\draw[fill] (\i) circle[radius=0.1];
}
}
\end{center}

Let $X$ be a set of $v$ points or vertices and let $\vec{K_v}$ denote the complete directed graph on vertex set $X$. This graph has $v(v-1)$ directed edges.
A  \emph{directed triple system of order $v$} is a pair $(X, \B)$, where 
$X$ is a set of $v$ \emph{points} and $\B$ is a set of transitive  triples (or more simply, \emph{triples}) whose elements are members of $X$,
 such that every directed  edge in $\vec{K_v}$ occurs  in exactly one triple in $\B$. (Thus, the triples in a directed triple system fulfill the same role as blocks in  a Steiner triple system.)
 
  We will abbreviate the phrase ``directed triple system of order $v$'' to 
\DTS$(v)$. It is well-known that a \DTS$(v)$ contains exactly $v(v-1)/3$ triples, and a 
\DTS$(v)$ exists if and only  if $v \equiv 0,1 \bmod 3$. Various results on \DTS$(v)$ can  be found in \cite{CR}.
 
The following problem on sequencing points in an \STS$(v)$ was introduced by Kreher and Stinson in \cite{KS2}
and studied further in Stinson and Veitch \cite{SV}.
Suppose $(X, \B)$ is an \STS$(v)$  and let $\ell \geq 3$ be an integer. 
 A \emph{sequencing} of the \STS$(v)$ is a permutation
$\pi = [x_1\; x_2 \;  \cdots \;  x_v]$ of $X$.
A sequencing  $\pi = [x_1\; x_2 \;  \cdots \;  x_v]$ is \emph{$\ell$-good} if no $\ell$ consecutive points in $\pi$ contain a block in $\B$.

Some related but different sequencing problems for \STS$(v)$ are studied in \cite{AKP} and \cite{KS}. Also, for a recent survey paper on this topic, see \cite{Alspach}.

It is obvious that an \STS$(v)$ cannot have a $v$-good sequencing. In fact, it was shown in \cite{SV} that,
if an $\STS(v)$ with $v \geq 7$ has an $\ell$-good sequencing,  then
$\ell < (v+2)/3.$

In this paper, we study the corresponding question for \DTS$(v)$. Let $(X, \B)$ be a \DTS$(v)$.
We first need to define what it means for a sequencing $\pi$ of $(X, \B)$ to ``contain'' a particular transitive triple.
The most natural approach seems to be to regard the sequencing as a total ordering defined on the points in $X$. 
A triple $(x,y,z) \in \B$ is said to be \emph{contained} in $\ell$ consecutive points of the sequencing 
$\pi = [x_1\; x_2 \;  \cdots \;  x_v]$ 
if 
\begin{enumerate}
\item $\{x,y,z\} \subseteq \{x_i, x_{i+1}, \dots , x_{i + \ell - 1}\}$ for some $i$, and 
\item $x < y < z$ in the sequencing.
\end{enumerate}
Then, a sequencing  $\pi$ is \emph{$\ell$-good} if no $\ell$ consecutive points in the sequencing contain a triple in $\B$.

Unlike \STS$(v)$, it is possible that a \DTS$(v)$ has a $v$-good sequencing. Informally, this just means that 
for every triple $(x,y,x)$ in the \DTS$(v)$, the ordering of $x$, $y$ and $z$ in the sequencing is \emph{not}
$x < y < z$. 

\begin{example} 
\label{3.exam}
{\rm Let $X = \{1,2,3\}$ and $\B = \{ (1,2,3), (3,2,1)\}$.
Then $(X,\B)$ is a \DTS$(3)$ and
 $[1\; 3\; 2]$, $[2\; 3\; 1]$, $[2\; 1\; 3]$ and $[3\; 1\; 2]$ are all 3-good sequencings.
 }
\end{example}

\begin{example} 
\label{4.exam}
{\rm Let $X = \{1,2,3,4\}$ and $\B = \{ (1,2,3), (2,1,4), (3,4,2), (4,3,1)\}$.
Then $(X,\B)$ is a \DTS$(4)$ and
 $[1\; 3\; 2\; 4]$ is a  4-good sequencing.
 }
\end{example}

\begin{example} 
\label{6.exam}
{\rm Let $X = \{\infty\} \cup \zed_5$ and $\B = \{ (0, \infty, 4), (0,1,3) \} \bmod 5$.
Then $(X,\B)$ is a \DTS$(6)$ and
 $[\infty\; 0\; 2\; 4\; 3\; 1]$ is a 6-good sequencing.
 }
\end{example}

\subsection{Summary of results}

In Section \ref{constructions.sec}, we use recursive constructions to prove that there exists a \DTS$(v)$ having a $v$-good sequencing for all positive integers $v \equiv 0,1 \bmod {3}$. 
In Section \ref{computation.sec}, we  report some computational results concerning  $v$-good sequencings of all the nonisomorphic \DTS$(v)$ for $v \leq 7$. Perhaps surprisingly, there are precisely four nonisomorphic 
\DTS$(7)$ (out of a total of 2368) that do not have $7$-good sequencings. 
In Section \ref{algorithms.sec}, we investigate a possible algorithmic approach to prove that a given 
\DTS$(v)$ does not have  a $v$-good sequencing. We illustrate by providing  a short  proof that a certain 
\DTS$(7)$ does not have a $7$-good sequencing. We also use the same technique to enable the construction of  \DTS$(v)$ that do not have  $v$-good sequencings for  $v = 9,10,12,13,16$ and $18$.
Then, in Section \ref{nonsequenceable.sec}, we use recursive constructions to prove that there is a \DTS$(v)$ that does not have  a $v$-good sequencing
for all positive integers $v \equiv 0,1 \bmod {3}$, $v \geq 7$.

\section{Constructions}
\label{constructions.sec}

We provide two proofs that there exists a \DTS$(v)$ having  a  $v$-good sequencing for all $v \equiv 0,1 \bmod 3$.
First, we give a PBD proof. Then we present a proof using two well-known recursive ``doubling'' constructions for DTS,
having the form $v \rightarrow 2v+1$ and $v \rightarrow 2v+4$.

\subsection{PBD-closure}

Let  $K$ be a set whose elements are all integers $\geq 2$.
A pair $(X, \B)$ is a \emph{$(v,K)$-pairwise balanced design} 
(or, \emph{$(v,K)$-PBD})
if $X$ is a set of $v$ \emph{points} and $\B$ is a set of subsets of $X$ (called 
\emph{blocks}) such that 
\begin{itemize}
\item every pair of points from $X$ occurs in exactly one block in $\B$, and
\item
$|B| \in K$ for every $B \in \B$.
\end{itemize}
A set $K$, whose elements are all integers $\geq 2$, is \emph{PBD-closed} if $v \in K$
whenever there exists a $(v,K)$-PBD.

Let $\K = \{v \geq 3: \text{there exists a \DTS$(v)$ having a  $v$-good sequencing}\}$.
We show that $\K$ is PBD-closed.

\begin{theorem}
\label{PBDclosed.thm}
$\K$ is PBD-closed.
\end{theorem}

\begin{proof}
Suppose $(X, \B)$ is a $(v,\K)$-PBD, where $X = \{1, \dots, v\}$. We will construct a  
\DTS$(v)$ having a $[1\; 2\; \cdots \; v]$ as a $v$-good sequencing.

Let $B \in \B$,  say $B = \{x_1, x_2, \dots, x_k\}$, where
$x_1 < \cdots < x_k$. There is \DTS$(k)$ having a $k$-good sequencing.
Therefore, by relabelling points, there exists a \DTS$(k)$, say $(B, \A_B)$, for which 
$[x_1 \; x_2 \; \cdots \; x_k]$ is a  $k$-good sequencing.

Define 
\[\A = \bigcup _{B \in \B} \A_B.\]
It is straightforward to verify that $(X, \A)$ is a \DTS$(v)$ for which 
$[1 \; 2 \; \cdots \; v]$ is a  $v$-good sequencing.
\end{proof}

\begin{corollary}
There exists a \DTS$(v)$ having  a  $v$-good sequencing if and only if $v \equiv 0,1 \bmod 3$.
\end{corollary}

\begin{proof}
We have already noted that $v \equiv 0,1 \bmod 3$ is a necessary condition for existence of a \DTS$(v)$.

We prove sufficiency as follows.
For $v \equiv 0,1 \bmod 3$, $v \geq 7$, there exists a $(v,\{3,4\})$-PBD (see \cite[Table IV.3.23]{CD}).
We know that $3,4 \in \K$ from Examples \ref{3.exam} and \ref{4.exam}.
Then we can apply Theorem \ref{PBDclosed.thm} to show that $v\in \K$.
Finally, $6 \in \K$ from Example \ref{6.exam}.
\end{proof}

\subsection{Doubling constructions}

In this section, we prove the existence of a $v$-good sequencing of a $\DTS(v)$ by using two doubling constructions. The two constructions we use can be found in 
\cite[\S24, Lemma 1.1 and Lemma 1.2]{CR}, for example.

\begin{lemma}
\label{double1.lem}
If there exists a $\DTS(v)$ having a $v$-good sequencing, 
then there exists a $\DTS(2v+1)$ having a $(2v+1)$-good sequencing.
\end{lemma}
\begin{proof}
Let $(X = \{1,\dots,v\}, \mathcal{B})$ be a $\DTS(v)$ having a $v$-good sequencing $[1\; 2\; \cdots\; v]$.
Let $L$ be a latin square of order $v+1$ having constant diagonal, whose rows and columns are indexed by the set $Y = \{v+1, \dots, 2v+1\}$ of size $v+1$ and whose off-diagonal symbols are from $X$. Form a set $\mathcal{C}$ of triples as follows: For each $i,j \in Y$, $i \neq j$, let $(i, L(i,j),j) \in \mathcal{C}$. 
Then, $(X \cup Y, \mathcal{B} \cup \mathcal{C})$ is a $\DTS(2v+1)$.

It is not hard to see that $(2v+1)$-good sequencing of this DTS is given by
\[
[1\; 2\; \cdots \; v\; v+1\; \cdots \; 2v+1].
\]
This follows, because 
\begin{enumerate}
\item $[1\; 2\; \cdots\; v]$ is a $v$-good sequencing of the triples in $\mathcal{B}$, and 
\item for each triple  $(i, L(i,j),j) \in \mathcal{C}$, the point $L(i,j)$ occurs in the sequencing before the point $i$.
\end{enumerate}
\end{proof}

\begin{lemma}
\label{double2.lem}
If there exists a $\DTS(v)$ having a $v$-good sequencing, then there exists a $\DTS(2v+4)$ having a $(2v+4)$-good sequencing.
\end{lemma}
\begin{proof}
Let $(X = \{x_1,\dots,x_v\}, \mathcal{B})$ be a $\DTS(v)$ having the $v$-good sequencing $[x_1\ x_2\ \cdots\ x_v]$.
Let $Y = \mathbb{Z}_{v+4}$ be disjoint from $X$. Form $v$ disjoint sets $S_1, \dots, S_v$, each consisting  of $v+4$ ordered pairs of points from $Y$, by taking \[S_i = \{(a,b) : b-a \equiv i \bmod{(v+4)}\}\]
for $i = 1, \dots , v$. Now form a set $\mathcal{C}$ comprised of the following triples:
\begin{enumerate}
	\item for each $i$, $1 \leq i \leq v$,  and for every $(a,b) \in S_i$, 
	the triple $(a,x_i,b) \in \mathcal{C}$, and
	\item for each $i \in \mathbb{Z}_{v+4}$, the triple $(i, v+2+i,v+1+i) \bmod{(v+4)} \in \mathcal{C}$.
\end{enumerate}
Then $(X \cup Y, \mathcal{B} \cup \mathcal{C})$ is a $\DTS(2v+4)$.

We claim that a $(2v+4)$-good sequencing of this DTS is given by
\[
[x_1\; x_2\; \cdots \; x_v\; 0\; 1\; \cdots \; v+3].
\]
The first $v$ points do not contain a triple because $[x_1\; x_2\; \cdots\; x_v]$ is a $v$-good sequencing of the triples in $\mathcal{B}$. For each of the triples $(a,x_i,b) \in \mathcal{C}$ constructed in 1., $x_i$ occurs in the sequencing before the point $a$. Also, for each of the triples in $(i, v+2+i,v+1+i) \in \mathcal{C}$ constructed in 2., 
either 
\begin{enumerate}
\item $v+2+i \bmod{(v+4)} > v+1+i \bmod{(v+4)}$, or 
\item $v+2+i = 0$, in which case $i > 0 = v+2+i$. 
\end{enumerate}
So, the sequencing is $(2v+4)$-good.
\end{proof}

These lemmas suffice to prove the desired existence result.

\begin{theorem}
There exists a $\DTS(v)$ having a $v$-good sequencing if and only if $v \equiv 0,1 \mod{3}$.
\end{theorem}
\begin{proof}
We have already noted  that $v \equiv 0,1 \bmod{3}$ is a necessary condition for existence of a $\DTS(v)$, and
there exists a $\DTS(v)$ with a $v$-good sequencing for $v = 3,4,$ and 6. We proceed by induction. Suppose $v > 6$, $v \equiv 0,1 \bmod{3}$. If $v$ is odd, write $v = 2k+1$. Then $k \equiv 0,1 \bmod{3}$, so by induction, it follows that there exists a $\DTS(k)$ having a $k$-good sequencing. Hence, there exists a $\DTS(v)$ having a $v$-good sequencing by applying Lemma \ref{double1.lem}. Similarly, if $v$ is even, write $v = 2k+4$ and apply Lemma \ref{double2.lem}.
\end{proof}

\section{Computational results}
\label{computation.sec}

In this section, we report our results on $v$-good sequencings of \DTS$(v)$, for $v \leq 7$. The 
nonisomorphic \DTS$(v)$ for $v \leq 7$ have been enumerated by Colbourn and Colbourn \cite{CC} (see also \cite{OP}). We can test a \DTS$(v)$
by exhaustively checking all $v!$ permutations to see which of them are $v$-good sequencings. This does not take very much time for these small values of $v$.

Up to isomorphism, there is a unique \DTS$(3)$ and it has a $3$-good sequencing, as shown in 
Example \ref{3.exam}.

There are three nonisomorphic \DTS$(4)$. We present the three designs, along with $4$-good sequencings:

\begin{description}
\item[\D{4}{1}{}]:	$\begin{array}{*{15}{l}}
 (0,3,2)& (1,2,3)&
	 (2,1,0)& (3,0,1)&\end{array}$

	$4$-good sequencing: 0213\\
	number of $4$-good sequencings: 8
\end{description}\begin{description}
\item[\D{4}{2}{}]:	$\begin{array}{*{15}{l}}
 (0,3,2)& (1,2,3)&
	 (2,0,1)& (3,1,0)&\end{array}$
	 
    $4$-good sequencing: 0213\\
	number of $4$-good sequencings: 8
\end{description}\begin{description}
\item[\D{4}{3}{}]:	$\begin{array}{*{15}{l}}
 (0,3,2)& (1,2,0)&
	 (2,1,3)& (3,0,1)&\end{array}$

    $4$-good sequencing: 0123\\
	number of $4$-good sequencings: 8
\end{description}

There are 32 nonisomorphic \DTS$(6)$ and they all have $6$-good sequencings. The designs and their $6$-good sequencings are presented in the technical report \cite{KSV}.

There are exactly 2368 nonisomorphic \DTS$(7)$. We construct these following the method described in \cite{OP}.
There are four nonisomorphic $(v,3,2)$-BIBDs (or \TTS$(7)$), which we denote \D{7}{1}{}, \D{7}{2}{}, \D{7}{3}{} and \D{7}{4}{}. The triples in these four designs are directed in all possible ways to form \DTS$(7)$ and then isomorphic designs are eliminated. It turns out that all but four of the nonisomorphic \DTS$(7)$ have $7$-good sequencings. These $7$-good sequencings are all presented in \cite{KSV}.

The results can be summarized as follows:
\begin{itemize}
\item 18 \DTS$(7)$ have \D{7}{1}{} as the underlying \TTS$(7)$. All of these \DTS$(7)$ have $7$-good sequencings.
\item 274 \DTS$(7)$ have \D{7}{2}{} as the underlying \TTS$(7)$.  All of these \DTS$(7)$ have $7$-good sequencings.
\item 1060 \DTS$(7)$ have \D{7}{3}{} as the underlying \TTS$(7)$.  All of these \DTS$(7)$ have $7$-good sequencings.
\item 1016 \DTS$(7)$ have \D{7}{4}{} as the underlying \TTS$(7)$.  1012 of these \DTS$(7)$ have $7$-good sequencings.
\end{itemize}

It is interesting to note that  the four \DTS$(7)$ that do not have 
 $7$-good sequencings all have  a $6$-good sequencing. 
 These four \DTS$(7)$, along with $6$-good sequencings,
 are as follows:

\begin{description}
\item[\D{7}{4.}{926}]:	$\begin{array}{*{15}{l}}
 (0,4,2)& (0,5,6)& (1,3,0)& (1,5,2)& (2,0,1)& (2,6,5)& (3,1,6)&\\
	 (3,2,4)& (4,3,5)& (4,6,1)& (5,0,3)& (5,1,4)& (6,2,3)& (6,4,0)&\end{array}$
	 
	 6-good sequencing: 0123456\\
	number of 6-good sequencings: 124
\end{description}

\begin{description}
\item[\D{7}{4.}{958}]:	$\begin{array}{*{15}{l}}
 (0,4,2)& (0,5,3)& (1,5,2)& (1,6,3)& (2,1,0)& (2,3,4)& (3,0,1)&\\
	 (3,2,6)& (4,3,5)& (4,6,0)& (5,0,6)& (5,1,4)& (6,2,5)& (6,4,1)&\end{array}$
	 
	 6-good sequencing: 0245613\\
	number of 6-good sequencings: 124
	
\end{description}

\begin{description}
\item[\D{7}{4.}{1015}]:	$\begin{array}{*{15}{l}}
 (0,3,5)& (0,4,2)& (1,2,5)& (1,6,3)& (2,1,0)& (2,3,6)& (3,0,1)&\\
	 (3,2,4)& (4,0,6)& (4,5,3)& (5,4,1)& (5,6,0)& (6,1,4)& (6,5,2)&\end{array}$
	 
	 6-good sequencing: 0153462\\
	number of 6-good sequencings: 112
\end{description}

\begin{description}
\item[\D{7}{4.}{1016}]:	$\begin{array}{*{15}{l}}
 (0,3,1)& (0,4,6)& (1,2,0)& (1,6,4)& (2,1,5)& (2,3,4)& (3,0,5)&\\
	 (3,2,6)& (4,0,2)& (4,5,1)& (5,4,3)& (5,6,2)& (6,1,3)& (6,5,0)&\end{array}$
	 
	 6-good sequencing: 0124356\\
	number of 6-good sequencings: 112
\end{description}

It does not seem feasible to test all the \DTS$(9)$ because it is shown in 
\cite{OP} that there are $596,893,386$ nonisomorphic \DTS$(9)$.

\section{Algorithmic approaches}
\label{algorithms.sec}

It is of interest to devise an algorithm to determine if  a given \DTS$(v)$
can be sequenced. Obviously, checking all $v!$ permutations is not practical for large values of $v$, 
so we would like to have a more efficient algorithm.

Here is one possible approach that could be considered.
A directed triple $(a,b,c)$ in a \DTS$(v)$ 
leads to the following necessary condition for the existence of a $v$-good sequencing of $v$ points:
\begin{equation}
\label{condition.eq}
(c < b) \vee (b < a).
\end{equation}
For each of the $v(v-1)/3$ triples in a \DTS$(v)$, we obtain a condition similar to
(\ref{condition.eq}). Suppose, for each of the triples, we choose one of the two relevant inequalities 
(i.e., for the triple $(a,b,c)$, we choose $c < b$, or we choose $b < a$). We can interpret an inequality as an edge in a directed graph, i.e., $c < b$ corresponds to the directed edge $(c,b)$ and 
$b < a$ corresponds to the directed edge $(b,a)$.
Thus we obtain by this method a directed graph $\mathcal{D}$, on the $v$ points of the \DTS$(v)$, having $e = v(v-1)/3$ edges. 

It is easy to determine in polynomial time if this graph $\mathcal{D}$ has a \emph{topological ordering} (i.e., a total ordering of the points that is compatible with all the edges in the directed graph). A topological ordering is clearly a $v$-good sequencing of the given \DTS$(v)$. It is well-known that there is a topological ordering of a directed graph if and only if the graph is a DAG (directed acyclic graph). 
Further, testing a directed graph to see if it has a topological ordering 
can be done using a simple modification of DFS (depth-first search).  
The complexity of this algorithm (given a particular graph $\mathcal{D}$) is $O(v+e) = O(v^2)$.
(For these results, see, for example, \cite[\S 22]{CLRS}.)

We could construct all the possible directed graphs and test each of them in this way. If none of the graphs are DAGs, then the \DTS$(v)$ does not have a $v$-good sequencing. The problem is that there are $2^{v(v-1)/3}$ graphs to test, so this is not a polynomial-time algorithm. 
However, in practice, we can often achieve a significant reduction in the number of graphs to be considered. 
It is possible that this approach might lead to a fairly simple proof that a given 
\DTS$(v)$ has no $v$-good sequencing. This technique works well in practice for small values of $v$ and it can even be done by hand with a bit of patience.  We illustrate by deriving a proof that the \DTS$(7)$ named 
\D{7}{4.}{926}
(which was presented in Section \ref{computation.sec}) has no $7$-good sequencing.

\begin{theorem} The \DTS$(7)$ named \D{7}{4.}{926} does not have a $7$-good sequencing.
\end{theorem}

\begin{proof}
First, we list the triples in \D{7}{4.}{926}, along with the conditions derived from them, in Table \ref{proof.tab}. The idea is to show that any directed graph that satisfies the required conditions for every triple must contain a directed cycle.
\begin{table}
\caption{The triples in $D(7)4.926$}
\label{proof.tab}
\[
\begin{array}{c|c}
\text{triple} & \text{condition} \\ \hline
T_1=(0,4,2) & (4 < 0) \vee (2 < 4)\\
T_2=(0,5,6) & (5 < 0) \vee (6 < 5)\\
T_3=(1,3,0) & (3 < 1) \vee (0 < 3)\\
T_4=(1,5,2) & (5 < 1) \vee (2 < 5)\\
T_5=(2,0,1) & (0 < 2) \vee (1 < 0)\\
T_6=(2,6,5) & (6 < 2) \vee (5 < 6)\\
T_7=(3,1,6) & (1 < 3) \vee (6 < 1)\\
T_8=(3,2,4) & (2 < 3) \vee (4 < 2)\\
T_9=(4,3,5) & (3 < 4) \vee (5 < 3)\\
T_{10}=(4,6,1) & (6 < 4) \vee (1 < 6)\\
T_{11}=(5,0,3) & (0 < 5) \vee (3 < 0)\\
T_{12}=(5,1,4) & (1 < 5) \vee (4 < 1)\\
T_{13}=(6,2,3) & (2 < 6) \vee (3 < 2)\\
T_{14}=(6,4,0) & (4 < 6) \vee (0 < 4)
\end{array}
\]
\end{table}
We begin by considering triple $T_1$. The proof divides into two cases:
\begin{description}
\item [Case 1]: $4 < 0$
\item [Case 2]: $2 < 4$
\end{description}

For case 1, we assume $4 < 0$ and we proceed as follows:
\begin{eqnarray*}
T_{14} &\implies& (4 < 6)  \vee  (0 < 4), \text{ so } 4 < 6\\
T_{10} &\implies &(6 < 4)  \vee  (1 < 6), \text{ so } 1 < 6\\
T_7 &\implies &(1 < 3)  \vee  (6 < 1), \text{ so }  1 < 3\\
T_3 &\implies &(3 < 1)  \vee  (0 < 3), \text{ so }  0 < 3\\
T_{11} &\implies& (0 < 5)  \vee  (3 < 0), \text{ so }  0 < 5\\
T_2 &\implies &(5 < 0)  \vee  (6 < 5), \text{ so }  6 < 5\\
T_6 &\implies &(6 < 2)  \vee  (5 < 6), \text{ so }  6 < 2\\
T_{13} &\implies &(2 < 6)  \vee  (3 < 2), \text{ so }  3 < 2\\
T_8 &\implies &(2 < 3)  \vee  (4 < 2), \text{ so }  4 < 2.
\end{eqnarray*}
So far, there are no directed cycles, so we proceed a bit further.
\begin{eqnarray*}
T_4 &\implies&  (5 < 1)  \vee  (2 < 5). 
\end{eqnarray*}
If $5 < 1$, then $1 < 6 < 5 < 1$ is a directed cycle.  
   Therefore $2<5$.
\begin{eqnarray*}
T_9 &\implies&  (3<4)  \vee  (5<3). 
\end{eqnarray*}
If $3 < 4$, then we get the directed cycle  $3 < 4 < 0 < 3$.
Therefore $5 < 3$. But this creates the directed cycle $5 < 3 < 2 < 5$.
Thus Case 1 is impossible.

\medskip

Now we turn to Case 2, where we assume $2 < 4$. We proceed as follows:
\begin{eqnarray*}
T_{8} &\implies& (2<3)  \vee  (4 < 2), \text{ so } 2<3\\
T_{13} &\implies &(2<6)  \vee  (3<2), \text{ so } 2 < 6\\
T_6 &\implies &(6<2)  \vee  (5<6), \text{ so }  5<6\\
T_2 &\implies &(5<0)  \vee  (6<5), \text{ so }  5<0\\
T_{11} &\implies& (0 < 5)  \vee  (3 < 0), \text{ so }  3<0\\
T_3 &\implies &(3<1)  \vee  (0<3), \text{ so }  3<1\\
T_7 &\implies &(1<3)  \vee  (6<1), \text{ so }  6 < 1\\
T_{10} &\implies &(6<4)  \vee  (1<6), \text{ so }  6 < 4\\
T_{14} &\implies &(4<6)  \vee  (0<4), \text{ so }  0<4.
\end{eqnarray*}
We continue.
\begin{eqnarray*}
T_{5} &\implies& (0<2)  \vee  (1<0).
\end{eqnarray*}
If $0 < 2$, then $3 < 0 < 2 < 3$ is a directed cycle. Therefore $1 < 0$.
\begin{eqnarray*}
T_{12} &\implies&  (1<5)  \vee  (4<1). 
\end{eqnarray*}
   If $1 < 5$, then $6 < 1 < 5 < 6$ is a directed cycle. Therefore $4 < 1$
   But then $0 < 4 < 1 < 0$ is a directed cycle.
Thus Case 2 is also impossible.
\end{proof}

Similar reasoning can be used to show that the \DTS$(7)$ named \D{7}{4.}{958}, 
\D{7}{4.}{1015}, and \D{7}{4.}{1016} 
(all of which are presented in Section \ref{computation.sec}) have no $7$-good sequencings.

We are also able to use this technique to construct \DTS$(v)$ that can be proven not to have a
$v$-good sequencing for $v \in \{9,10,12,13,16,18\}$.  Our proof  depends on the following lemma.

\begin{lemma}
\label{Shannon.lem}
Suppose that a \DTS$(v)$ contains the following twelve triples:
\[
\begin{array}{llllll}
(1,2,3)
&(4,3,2)
&(3,4,5)
&(6,5,4)
&(5,6,2)
&(7,2,6)\\
(2,7,8)
&(3,8,7)
&(8,3,6)
&(2,1,0)
&(6,0,1)
&(0,6,3)
\end{array}
\]
Then the \DTS$(v)$  cannot have a $v$-good sequencing.
\end{lemma}

\begin{table}[tb]
\caption{Twelve triples}
\label{Shannon.tab}
\begin{center}
\begin{tabular}{c | c}
	\text{triple} & \text{condition} \\
	\hline
	$T_1 = (1,2,3)$ & $(2 < 1) \vee (3 < 2)$ \\
	$T_2 = (4,3,2)$ & $(3 < 4) \vee (2 < 3)$ \\
	$T_3 = (3,4,5)$ & $(4 < 3) \vee (5 < 4)$ \\
	$T_4 = (6,5,4)$ & $(5 < 6) \vee (4 < 5)$ \\
	$T_5 = (5,6,2)$ & $(6 < 5) \vee (2 < 6)$ \\
	$T_6 = (7,2,6)$ & $(2 < 7) \vee (6 < 2)$ \\
	$T_7 = (2,7,8)$ & $(7 < 2) \vee (8 < 7)$ \\
	$T_8 = (3,8,7)$ & $(8 < 3) \vee (7 < 8)$ \\
	$T_9 = (8,3,6)$ & $(3 < 8) \vee (6 < 3)$ \\
	$T_{10} = (2,1,0)$ & $(1 < 2) \vee (0 < 1)$ \\
	$T_{11} = (6,0,1)$ & $(0 < 6) \vee (1 < 0)$ \\
	$T_{12} = (0,6,3)$ & $(6 < 0) \vee (3 < 6)$
\end{tabular}
\end{center}
\end{table}

\begin{proof}

First, we list the twelve triples, along with the conditions derived from them, in Table \ref{Shannon.tab}. We begin by considering triple $T_1$. The proof divides into two cases:

\begin{description}
\item [Case 1]: $3 < 2$
\item [Case 2]: $2 < 1$
\end{description}

For Case 1, we assume $3 < 2$. We proceed as follows:
\begin{align*}
T_{2} &\implies (3 < 4) \vee (2 < 3)\text{, so } 3 < 4 \\
T_{3} &\implies (4 < 3) \vee (5 < 4)\text{, so } 5 < 4 \\
T_{4} &\implies (5 < 6) \vee (4 < 5)\text{, so } 5 < 6 \\
T_{5} &\implies (6 < 5) \vee (2 < 6)\text{, so } 2 < 6 \\
T_{6} &\implies (2 < 7) \vee (6 < 2)\text{, so } 2 < 7 \\
T_{7} &\implies (7 < 2) \vee (8 < 7)\text{, so } 8 < 7 \\
T_{8} &\implies (8 < 3) \vee (7 < 8)\text{, so } 8 < 3 \\
T_{9} &\implies (3 < 8) \vee (6 < 3)\text{, so } 6 < 3.
\end{align*}
But then $3 < 2 < 6 < 3$ is a directed cycle. Thus Case 1 is impossible.

\medskip

Now we consider Case 2, where we assume $2 < 1$. Then we proceed as follows:
\begin{align*}
T_{10} &\implies (1 < 2) \vee (0 < 1)\text{, so } 0 < 1 \\
T_{11} &\implies (0 < 6) \vee (1 < 0)\text{, so } 0 < 6 \\
T_{12} &\implies (6 < 0) \vee (3 < 6)\text{, so } 3 < 6 \\
T_{9} &\implies (3 < 8) \vee (6 < 3)\text{, so } 3 < 8 \\
T_{8} &\implies (8 < 3) \vee (7 < 8)\text{, so } 7 < 8 \\
T_{7} &\implies (7 < 2) \vee (8 < 7)\text{, so } 7 < 2 \\
T_{6} &\implies (2 < 7) \vee (6 < 2)\text{, so } 6 < 2 \\
T_{5} &\implies (6 < 5) \vee (2 < 6)\text{, so } 6 < 5 \\
T_{4} &\implies (5 < 6) \vee (4 < 5)\text{, so } 4 < 5 \\
T_{3} &\implies (4 < 3) \vee (5 < 4)\text{, so } 4 < 3 \\
T_{2} &\implies (3 < 4) \vee (2 < 3)\text{, so } 2 < 3. 
\end{align*}
But then $3 < 6 < 2 < 3$ is a directed cycle. Thus Case 2 is also impossible.
\end{proof}

For $v \in \{9,10,12,13,16,18\}$, we have constructed \DTS$(v)$ that contain the twelve triples listed in 
Lemma \ref{Shannon.lem}; see Examples \ref{DTS9.ex}--\ref{DTS18.ex}. The construction of these \DTS$(v)$ made use of a hill-climbing algorithm that is similar to the hill-climbing algorithm to construct \STS$(v)$ that is presented in \cite{Stinson}.

We provide a brief description of the hill-climbing approach we used. The algorithm attempts to construct a
\DTS$(v)$ by using three heuristics, which we name $H_1$, $H_2$ and $H_3$. In the following, $x$, $y$ and $z$ refer to points in the  \DTS$(v)$ that we are constructing.

\begin{center}
\noindent\framebox{
\begin{minipage}{4.75in}
\begin{description}[style=unboxed,leftmargin=0cm]
\item [$H_1$]  \quad If there exists a point $x$ such that there are at least two points $y,z$ such that the directed edges $xy$ and $xz$ have not occurred in a triple, then construct the triple $(x,y,z)$ and add it to the design. 
If the directed edge $yz$ already appears in a triple, then delete that triple.
\end{description}
\end{minipage}
}
\end{center}

\begin{center}
\noindent\framebox{
\begin{minipage}{4.75in}
\begin{description}[style=unboxed,leftmargin=0cm]
\item [$H_2$]    \quad If there exists a point $x$ such that there are two points $y,z$ such that the directed edges $yx$ and $xz$ have not occurred in a triple, then construct the triple $(y, x, z)$ and add it to the design. If the directed edge $yz$ already appears in a triple, then delete that triple.
\end{description}
\end{minipage}
}
\end{center}

\begin{center}
\noindent\framebox{
\begin{minipage}{4.75in}
\begin{description}[style=unboxed,leftmargin=0cm]

\item [$H_3$]    \quad If there exists a point $x$ such that there are at least two points $y,z$ such that the directed edges $yx$ and $zx$ have not occurred in a triple, then construct the triple $(y,z,x)$ and add it to the design. If $yz$ already appears in a triple then delete that triple.
\end{description}
\end{minipage}
}
\end{center}

The hill-climbing algorithm would randomly apply $H_1$, $H_2$ and $H_3$ over and over again, until (hopefully) a design is constructed.
However, we are trying to do something a bit more complicated, namely, to construct a 
\DTS$(v)$ that contains  the twelve \emph{initial triples} listed in 
Lemma \ref{Shannon.lem}.  Thus, we begin with the initial triples and we need to modify $H_1$, $H_2$ and $H_3$ so that we never delete an initial triple. This is straightforward; for example, $H_1$ would be replaced by the following modified heuristic.

\begin{center}
\noindent\framebox{
\begin{minipage}{4.75in}
\begin{description}[style=unboxed,leftmargin=0cm]
\item [$H_1^*$]   \quad Suppose there exists a point $x$ such that there are at least two points $y,z$ such that the directed edges $xy$ and $xz$ have not occurred in a triple. 
\begin{enumerate}
\item If there is no existing triple containing the directed edge $yz$,  then construct the triple $(x,y,z)$ and add it to the design. 
\item If there is a non-initial triple containing the directed edge $yz$, then delete that triple and add the 
triple $(x,y,z)$ to the design. 
\item If there is an initial triple containing the directed edge $yz$, then do nothing
(the heuristic fails in this case).
\end{enumerate}
\end{description}
\end{minipage}
}
\end{center}

$H_2$ and $H_3$ would be modified in a similar fashion as $H_1$.

As we mentioned above, we used this hill-climbing algorithm  to find several \DTS$(v)$ that do not have 
$v$-good sequencings. It should be emphasized that the algorithm is very fast and it ran almost instantaneously on a laptop for the small designs we considered.

\begin{example}
\label{DTS9.ex} A $\DTS(9)$ that has no $9$-good sequencing.
\[
\begin{array}{l@{\hspace{.08in}}l@{\hspace{.08in}}l@{\hspace{.08in}}l@{\hspace{.08in}}l@{\hspace{.08in}}l@{\hspace{.08in}}l@{\hspace{.08in}}l}
(1,2,3)
&(4,3,2)
&(3,4,5)
&(6,5,4)
&(5,6,2)
&(7,2,6)
&(2,7,8)
&(3,8,7)\\
(8,3,6)
&(2,1,0)
&(6,0,1)
&(0,6,3)
&(1,7,5)
&(5,3,1)
&(0,2,4)
&(5,0,7)\\
(8,2,5)
&(1,6,8)
&(7,3,0)
&(0,5,8)
&(8,4,0)
&(4,6,7)
&(4,8,1)
&(7,1,4)
\end{array}
\]
\end{example}

\begin{example}
\label{DTS10.ex} A $\DTS(10)$ that has no $10$-good sequencing.
\[
\begin{array}{l@{\hspace{.08in}}l@{\hspace{.08in}}l@{\hspace{.08in}}l@{\hspace{.08in}}l@{\hspace{.08in}}l@{\hspace{.08in}}l@{\hspace{.08in}}l}
(1,2,3)
&(4,3,2)
&(3,4,5)
&(6,5,4)
&(5,6,2)
&(7,2,6)
&(2,7,8)
&(3,8,7)\\
(8,3,6)
&(2,1,0)
&(6,0,1)
&(0,6,3)
&(0,9,7)
&(7,3,1)
&(0,2,4)
&(1,8,4)\\
(1,5,7)
&(9,2,5)
&(4,1,9)
&(7,9,4)
&(9,3,0)
&(5,0,8)
&(6,9,8)
&(9,1,6)\\
(8,2,9)
&(5,3,9)
&(8,5,1)
&(4,8,0)
&(4,6,7)
&(7,0,5)
\end{array}
\]
\end{example}

\begin{example}
\label{DTS12.ex} A $\DTS(12)$ that has no $12$-good sequencing.
\[
\begin{array}{l@{\hspace{.08in}}l@{\hspace{.08in}}l@{\hspace{.08in}}l@{\hspace{.08in}}l@{\hspace{.08in}}l@{\hspace{.08in}}l}
(1,2,3)
&(4,3,2)
&(3,4,5)
&(6,5,4)
&(5,6,2)
&(7,2,6)
&(2,7,8)\\
(3,8,7)
&(8,3,6)
&(2,1,0)
&(6,0,1)
&(0,6,3)
&(11,4,7)
&(11,5,3)\\
(6,7,9)
&(11,0,2)
&(4,1,8)
&(5,0,7)
&(11,1,6)
&(6,11,10)
&(5,11,8)\\
(9,4,6)
&(7,4,10)
&(10,4,0)
&(0,9,5)
&(9,0,8)
&(10,7,5)
&(2,10,11)\\
(7,3,0)
&(3,11,9)
&(9,7,1)
&(10,3,1)
&(8,1,4)
&(8,0,10)
&(8,2,5)\\
(10,9,2)
&(10,6,8)
&(9,3,10)
&(5,1,10)
&(0,4,11)
&(2,4,9)
&(1,5,9)\\
(8,9,11)
&(1,7,11)
\end{array}
\]
\end{example}

\begin{example}
\label{DTS13.ex} A $\DTS(13)$ that has no $13$-good sequencing.
\[
\begin{array}{l@{\hspace{.08in}}l@{\hspace{.08in}}l@{\hspace{.08in}}l@{\hspace{.08in}}l@{\hspace{.08in}}l@{\hspace{.08in}}l}
(1,2,3)
&(4,3,2)
&(3,4,5)
&(6,5,4)
&(5,6,2)
&(7,2,6)
&(2,7,8)\\
(3,8,7)
&(8,3,6)
&(2,1,0)
&(6,0,1)
&(0,6,3)
&(10,4,6)
&(7,1,4)\\
(0,11,8)
&(6,11,12)
&(9,5,10)
&(10,8,2)
&(12,1,11)
&(1,10,5)
&(4,12,8)\\
(7,3,11)
&(10,11,3)
&(8,12,5)
&(11,9,6)
&(6,8,10)
&(0,12,2)
&(5,7,0)\\
(4,10,0)
&(2,9,12)
&(6,7,9)
&(7,5,12)
&(9,11,7)
&(10,1,7)
&(9,2,4)\\
(11,4,1)
&(12,4,7)
&(1,12,6)
&(9,1,8)
&(5,9,3)
&(12,10,9)
&(5,8,1)\\
(0,4,9)
&(8,9,0)
&(3,1,9)
&(0,7,10)
&(3,10,12)
&(8,4,11)
&(12,3,0)\\
(11,0,5)
&(2,5,11)
&(11,2,10)
\end{array}
\]
\end{example}

\begin{example}
\label{DTS16.ex} A $\DTS(16)$ that has no $16$-good sequencing.
\[
\begin{array}{l@{\hspace{.06in}}l@{\hspace{.06in}}l@{\hspace{.06in}}l@{\hspace{.06in}}l@{\hspace{.06in}}l@{\hspace{.06in}}l}
(1,2,3)
&(4,3,2)
&(3,4,5)
&(6,5,4)
&(5,6,2)
&(7,2,6)
&(2,7,8)\\
(3,8,7)
&(8,3,6)
&(2,1,0)
&(6,0,1)
&(0,6,3)
&(8,10,13)
&(4,14,1)\\
(7,0,9)
&(14,10,12)
&(1,14,6)
&(9,3,13)
&(5,1,9)
&(2,15,12)
&(9,8,12)\\
(12,13,8)
&(11,10,2)
&(14,8,4)
&(10,1,8)
&(13,12,3)
&(3,9,11)
&(7,4,10)\\
(13,6,15)
&(2,4,13)
&(14,7,5)
&(9,2,10)
&(15,14,13)
&(0,10,4)
&(6,11,12)\\
(12,5,15)
&(15,6,8)
&(5,3,12)
&(7,11,13)
&(13,11,1)
&(6,13,7)
&(4,15,7)\\
(1,10,11)
&(8,2,5)
&(13,5,10)
&(13,14,0)
&(12,9,7)
&(12,6,10)
&(5,14,11)\\
(12,1,4)
&(8,9,1)
&(10,9,15)
&(10,5,7)
&(10,3,0)
&(14,15,9)
&(10,6,14)\\
(15,0,5)
&(1,7,15)
&(8,0,14)
&(11,9,6)
&(13,9,4)
&(0,13,2)
&(2,9,14)\\
(9,5,0)
&(4,6,9)
&(4,12,0)
&(1,5,13)
&(15,3,10)
&(15,2,11)
&(11,15,4)\\
(3,15,1)
&(0,8,15)
&(12,14,2)
&(11,3,14)
&(11,5,8)
&(7,14,3)
&(4,8,11)\\
(0,12,11)
&(11,0,7)
&(7,1,12)
\end{array}
\]
\end{example}

\begin{example}
\label{DTS18.ex} A $\DTS(18)$ that has no $18$-good sequencing.
\[
\begin{array}{l@{\hspace{.06in}}l@{\hspace{.06in}}l@{\hspace{.06in}}l@{\hspace{.06in}}l@{\hspace{.06in}}l@{\hspace{.06in}}l}
(1,2,3)
&(4,3,2)
&(3,4,5)
&(6,5,4)
&(5,6,2)
&(7,2,6)
&(2,7,8)\\
(3,8,7)
&(8,3,6)
&(2,1,0)
&(6,0,1)
&(0,6,3)
&(6,11,16)
&(9,0,17)\\
(11,15,3)
&(10,12,8)
&(3,15,11)
&(15,14,13)
&(17,15,2)
&(11,7,4)
&(5,16,13)\\
(10,2,17)
&(17,3,14)
&(0,9,10)
&(12,13,2)
&(13,1,8)
&(2,4,10)
&(7,10,3)\\
(12,7,0)
&(6,9,7)
&(12,15,10)
&(1,17,16)
&(8,2,13)
&(2,9,16)
&(2,12,5)\\
(7,11,13)
&(12,17,6)
&(13,11,9)
&(16,9,8)
&(6,17,13)
&(16,11,2)
&(13,12,3)\\
(10,13,5)
&(13,0,16)
&(4,0,11)
&(3,0,13)
&(11,5,1)
&(17,1,7)
&(8,1,11)\\
(15,7,1)
&(16,7,14)
&(14,5,0)
&(9,15,4)
&(14,11,6)
&(14,16,3)
&(4,1,15)\\
(0,14,7)
&(10,4,9)
&(10,14,1)
&(9,5,3)
&(9,13,14)
&(1,13,10)
&(17,10,11)\\
(14,8,10)
&(7,15,9)
&(17,0,5)
&(11,10,0)
&(5,10,7)
&(5,11,12)
&(0,8,12)\\
(2,11,14)
&(16,17,4)
&(15,16,12)
&(8,0,4)
&(5,14,15)
&(4,8,16)
&(16,1,5)\\
(3,10,16)
&(6,15,8)
&(6,12,14)
&(4,17,12)
&(13,4,7)
&(3,12,1)
&(16,6,10)\\
(0,2,15)
&(9,1,6)
&(10,15,6)
&(1,4,14)
&(14,12,4)
&(7,12,16)
&(13,15,17)\\
(3,17,9)
&(16,15,0)
&(7,5,17)
&(14,9,2)
&(8,14,17)
&(8,15,5)
&(4,13,6)\\
(11,17,8)
&(5,8,9)
&(1,9,12)
&(12,9,11)
\end{array}
\]
\end{example}

\section{Existence of \DTS$(v)$ without $v$-good sequencings}
\label{nonsequenceable.sec}

Let $\K^* = \{v \geq 3: \text{there exists a \DTS$(v)$ having no  $v$-good sequencing}\}$.
In this section, we prove that $v \in \K^*$ 
for all $v \equiv 0,1 \bmod 3$, $v \geq 7$.

We summarize  results from Sections \ref{computation.sec} and \ref{algorithms.sec} in the following lemma.

\begin{lemma}
\label{small.lem}
$3,4,6 \not\in \K^*$ and $7,9,10,12,13,16,18 \in \K^*$.
\end{lemma}

Suppose $(Y,\B)$ is a \DTS$(w)$ and  $(X,\A)$ is a \DTS$(v)$. We say that  $(Y,\B)$ is a \emph{subdesign} of   $(X,\A)$  if $Y \subseteq X$ and $\B \subseteq \A$.
The following lemma is obvious. 

\begin{lemma}
\label{subdesign.lemma}
Suppose that a \DTS$(w)$ that does not have $w$-good sequencing is a 
subdesign of a \DTS$(v)$. Then the \DTS$(v)$
does not have $v$-good sequencing.
\end{lemma} 

\begin{theorem} 
\label{subdesign.thm}
Let $L = \{v \geq 3: v \equiv 0,1 \bmod 3\}$. Suppose $(X,\B)$ is a $(v,L)$-PBD and
suppose there exists a block $B_0 \in \B$ in the PBD such that $|B_0| \in \K^*$.
Then $v \in \K^*$.
\end{theorem}

\begin{proof}
Replace every block $B$ of the PBD by a \DTS$(|B|)$. For the block  $B_0$, fill in a 
\DTS$(|B_0|)$ that does not have a $|B_0|$-good squencing. 
The result  follows from Lemma \ref{subdesign.lemma}.
\end{proof}

\begin{corollary} 
\label{1,3mod6.cor}
Suppose $v \equiv 1,3 \bmod 6$, $v \geq 7$. Then $v \in \K^*$.
\end{corollary}
\begin{proof}
The values $v = 7, 9$   and $13$ are handled in Lemma \ref{small.lem}.
From the Doyen-Wilson Theorem \cite{DW}, 
there is an \STS$(v)$ that contains an \STS$(7)$ as a subdesign for all $v \geq 15$, $v \equiv 1,3 \bmod 6$. 
Replace the subdesign by a block of size 7, obtaining a $(v, \{3,7\})$-PBD that contains a (unique) block of size 7. Because $7 \in \K^*$, the result  follows from Theorem \ref{subdesign.thm}.
\end{proof}

\begin{corollary}
\label{0,4mod6.cor}
 Suppose $v \equiv 0,4 \bmod 6$, $v \geq 10$. Then $v \in \K^*$.
\end{corollary}

\begin{proof}
The values $v = 10,12,16$   and $18$ are handled in Lemma \ref{small.lem}.
For  $v \equiv 0,4 \bmod 6$, $v \geq 22$, write $v$ in the form  $v = 18s+r$, 
where  $r \in \{4,6,10,12,16,18\}$ and $s \geq 1$.
Because $v \geq 22$, we observe that
\begin{eqnarray*} 
18s +4 & = & 3 (6s+1) + 1\\
18s +6 & = & 3 (6s+1) + 3\\
18s +10& = & 3 (6s+3) + 1\\
18s +12& = & 3 (6s+3) + 3\\
18s +16& = & 3 (6s+3) + 7\\
18s +18 & = & 3 (6s+3) + 9.
\end{eqnarray*}
The first few equations in this series are $22 = 3 \times 7 + 1$, $24 = 3 \times 7 + 3$, $28 = 3 \times 9 + 1$
and $30 = 3 \times 9 + 3$.
Thus, it is clear that we can express $v$ in the form $v = 3m+t$, where 
$m \equiv 1,3 \bmod 6$, $m \geq 7$, $m \geq t$ and  and $t \in \{1,3,7,9\}$. 

Now, take a transversal design TD$(4,m)$ (see \cite{CD}) and delete $m-t$ points from one group. 
This gives rise to a $(v,\{ 3,4,m,t\})$-PBD that contains a block of size $m$. 
Corollary \ref{1,3mod6.cor} proves that $m \in \K^*$, so the 
desired result follows from Theorem \ref{subdesign.thm}.
\end{proof}

Summarizing the results proven in Corollaries \ref{1,3mod6.cor} and \ref{0,4mod6.cor} and Lemma \ref{small.lem}, we have the following.

\begin{theorem}
\label{main.thm}
Suppose $v \equiv 0 ,1 \bmod 3$, $v \geq 3$. Then $v \not\in \K^*$ if 
$v = 3,4$ or $6$ and $v \in \K^*$ if $v \geq 7$.
\end{theorem}

\section{Discussion and summary}

An interesting open question is if there is an efficient (i.e., polynomial-time) algorithm (perhaps using 
the ideas discussed in Section \ref{algorithms.sec}) to test if a given
\DTS$(v)$ has a $v$-good sequencing.

It would also be of interest to determine the proportion of \DTS$(v)$ having a $v$-good sequencing among all
\DTS$(v)$ of a given order $v$. We ask if this ratio approaches $1$ as $v$ increases.  

Even for $v=9$, there are too many nonisomorphic designs to test them all. However, we did generate $10000$ \DTS$(9)$ using our hill-climbing algorithm, and we determined that all but one of them has a $9$-good sequencing (this exceptional design has an $8$-good sequencing).  For $v=10$, we again generated $10000$ designs  using our hill-climbing algorithm, and we found that they all have a $10$-good sequencing. 

This suggests the following question: Does every \DTS$(v)$ have either a $v$-good sequencing or a $(v-1)$-good sequencing?


\begin{thebibliography}{X}

\bibitem{Alspach}
B.\ Alspach.
Variations on the sequenceable theme.
In ``50 Years of Combinatorics,
Graph Theory, and Computing'', 
F.\ Chung, R.\ Graham, F.\ Hoffman, L.\ Hogben, R.C.\ Mullin,
D.B.\ West, eds. CRC Press, 2020, to appear.

\bibitem{AKP}
B.\ Alspach, D.L.\ Kreher and A.\ Pastine.
Sequencing partial Steiner triple systems. Preprint.

\bibitem{CC} M.J.\ Colbourn and C.J.\ Colbourn. Some small directed triple systems.
\emph{Congr.\ Numer.} {\bf 30} (1981), 247--255.


\bibitem{CD}
C.J.\ Colbourn and J.H.\ Dinitz.
{\em Handbook of Combinatorial Designs, Second Edition},
 Chapman \& Hall/CRC, 2006.
 
\bibitem{CR}
C.J.\ Colbourn and A.\ Rosa. \emph{Triple Systems}, Oxford University Press, 1999.

\bibitem{CLRS}
T.H.\ Cormen, C.E.\ Leiserson, R.L.\ Rivest, and C.\ Stein. 
{\em Introduction to Algorithms, Third Edition.}
MIT Press, 2009. 

\bibitem{DW}
J.\ Doyen and R.M.\ Wilson. Embeddings of Steiner triple systems.
{\it Discrete Math.} {\bf 5} (1973), 229--239.

\bibitem{KS}
D.L.\ Kreher and D.R.\ Stinson.
Nonsequenceable Steiner triple systems.
\emph{Bull.\ Inst.\ Combin.\ Appl.\ } \textbf{86} (2019), 64--68.

\bibitem{KS2}
D.L.\ Kreher and D.R.\ Stinson.
Block-avoiding sequencings of points in Steiner triple systems.
\emph{Australas.\ J.\ Combin.\ } \textbf{74} (2019), 498--509.

\bibitem{KSV}
D.L.\ Kreher, D.R.\ Stinson  and S.\ Veitch.
Good sequencings for small directed triple systems.
Preprint.

\bibitem{OP}
P.R.J. \"{O}sterg\aa rd and O.\ Pottone.
Classification of directed and hybrid triple systems.
\emph{Bayreuth.\ Math.\ Schr.\ } {\bf 74} (2005), 276--291.

\bibitem{Stinson}
D.R.\ Stinson.
Hill-climbing algorithms for the construction of combinatorial
designs.
In ``Algorithms in Combinatorial Design Theory'', 
North-Holland, 1985, pp.\ 321--334
(\emph{Ann.\ Discrete Math.}, vol.\ 26).

\bibitem{SV}
D.R.\ Stinson and S.\ Veitch.
Block-avoiding point sequencings of arbitrary length in Steiner triple systems.
Preprint.



\end{thebibliography}
\end{document}